\documentclass[12pt]{amsart}
\usepackage{graphicx} 

\usepackage{amssymb, amsthm, enumerate, enumitem, graphicx, amscd, amsmath, mathrsfs, amsfonts, xcolor, comment}
\usepackage{hyperref}
\hypersetup{
colorlinks=true,
linktoc=all,
linkcolor=black,
citecolor=darkgray
}

\usepackage[mathcal]{euscript}
\usepackage{enumerate}
\usepackage{geometry}
\newtheorem{theorem}{Theorem}[section]
\newtheorem{theorem2}{Theorem} 

\newtheorem{lemma}[theorem]{Lemma}
\newtheorem{corollary}[theorem]{Corollary}

\newtheorem{proposition}[theorem]{Proposition}
\newtheorem{theorem*}{Theorem}

\theoremstyle{definition}

\theoremstyle{remark}

\theoremstyle{remark}
\newtheorem*{remark}{Remark}

\title[De Branges--Rovnyak and Dirichlet spaces]{On the equality of De
  Branges--Rovnyak and Dirichlet spaces}
\author[E.\ Dellepiane]{Eugenio Dellepiane}
\address{Dipartimento di Scienze Matematiche “Giuseppe Luigi Lagrange”, Politecnico di Torino, Corso Duca degli Abruzzi 24, 10129 Torino, Italy}
\email{eugenio.dellepiane@polito.it}

\author[M.\ M.\ Peloso]{Marco M. Peloso}
\address{Dipartimento di Matematica “Federigo Enriques", Università degli studi di Milano, Via Cesare Saldini 50, 20133 Milano, Italy}
\email{marco.peloso@unimi.it}

\author[A.\ Tabacco]{Anita Tabacco}
\address{Dipartimento di Scienze Matematiche “Giuseppe Luigi Lagrange”, Politecnico di Torino, Corso Duca degli Abruzzi 24, 10129 Torino, Italy}
\email{anita.tabacco@polito.it}

\date{}
\thanks{The authors are members of Gruppo Nazionale per l’Analisi
  Matematica, la Probabilit\`a e le loro Applicazioni (GNAMPA) of
  Istituto Nazionale di Alta Matematica (INdAM). The second author was
  partially supported by the 2025 INdAM–GNAMPA project {\em Function Spaces and Applications} (CUP\_E5324001950001).}
\begin{document}

\begin{abstract}
 This work is devoted to the comparison of de Branges--Rovnyak spaces $H(b)$ and harmonically weighted Dirichlet spaces $\mathcal{D}_\mu$. We completely characterize which $H(b)$ spaces are also harmonically
weighted Dirichlet spaces $\mathcal{D}_\mu$, when  $\mu$ is a finite sum of atoms. This is a generalization of a previous result by
Costara--Ransford \cite{costara2013}: we make no assumptions on the Pythagorean pair $(b,a)$, and we produce new examples.
\end{abstract}
\maketitle
\section{Introduction}
The de Branges--Rovnyak spaces $H(b)$ are a class of spaces of analytic functions on the unit disk $\mathbb{D}$ of the complex plain that serve as a fundamental bridge between
function theory and operator theory. While highly technical in nature, these spaces provide essential tools for solving concrete problems in complex analysis. Their study offers profound insights into the structure of contractions on Hilbert spaces, making them indispensable in modern analysis despite their abstract formulation.

The $H(b)$ spaces were introduced by de Branges and Rovnyak \cite{Debrangesrovnyak}, as a generalization of the complementary of the range space for the multiplication operator~$T_b$. Here, $b$ is an
analytic function on $\mathbb{D}$ that is uniformly bounded by $1$. Ever since their definition, the $H(b)$ spaces have attracted a lot of interest, because of their numerous applications to both operator theory and complex analysis. They provide canonical model
spaces for certain types of contractions on Hilbert spaces \cite{nagyfoias,Timotin}, and they were instrumental in the  proof of the Bieberbach conjecture  by de Branges, \cite{bieberbachdebranges}. 

In \cite{sarHb} Sarason started a new approach to this theory and suggested a new equivalent definition for the $H(b)$ spaces. In this formulation, the de Branges--Rovnyak spaces are realized as the range of the square root of the positive operator $I-T_b T_b^*$. They
are subsets of the Hardy space $H^2$ that are not necessarily closed in the norm topology of $H^2$. Many questions about the $H(b)$ spaces, even some apparently obvious ones, have been open and of interest in
the last decades, see e.g.\ the  monographs  \cite{hb1, hb2,sarHb} entirely devoted to its study. See also the nice introductory paper \cite{Timotin}. The main problem is that in general there is no direct way to establish whether a function $f$ belongs to an $H(b)$ space. Too little is know on the structure of
general $H(b)$ spaces. An explicit description of their elements is known only for a very limited number of cases. For example, in \cite{concrete} Fricain, Hartmann and Ross gave an explicit description of $H(b)$ when $b$ is of the form  $b=q^r,$ $q$ being a
rational outer function and $r>0$, and in \cite{FricainGrivaux} Fricain and Grivaux studied the case when $b$ has the form $b=(1+u)/2$ for some inner function~$u$.

A significant fact that helps us understanding the structure
of certain $H(b)$ spaces is the following: in some cases, $H(b)$
spaces are also weighted Dirichlet spaces, and in this latter case,
there is an explicit definition of the space through an integral. This
connection is particularly important because it transforms a highly
abstract 
space into one with a concrete integral representation, allowing us to
determine precisely which functions belong to the space. This bridge
between abstract operator 
theory and concrete function theory is what makes the relationship
between $H(b)$ spaces and weighted Dirichlet spaces valuable.

Given a positive finite Borel measure $\mu$ on $\mathbb{T}$, the \emph{harmonically weighted} Dirichlet space $\mathcal{D}_\mu$ is defined as
\[\mathcal{D}_\mu=\big\{f\in\operatorname{Hol}(\mathbb{D})\colon \mathcal{D}_\mu(f):=\int_\mathbb{D} |f'(z)|^2P_\mu(z) \operatorname{d}\!A(z)<\infty\big\}, \]
where $\operatorname{d}\!A$ is the area measure in $\mathbb{D}$, \[P_\mu(z)=\int_{\mathbb{T}}\frac{1-|z|^2}{|z-\zeta|^2}\operatorname{d}\!\mu(\zeta), \qquad z\in\mathbb{D},\]
is the Poisson integral of $\mu$ and $\mathbb{T}:=\partial\mathbb{D}$ is the unit circle.

They have been introduced by Richter in \cite{Richter1991}, in the context of the representation of cyclic analytic 2-isometries. They also play an important role in the study of the forward shift $S$ on the classic Dirichlet space $\mathcal{D}$: they appear in a very beautiful Beurling-type classification theorem for closed $S$-invariant subspaces of $\mathcal{D}$ \cite{RS1}. An important example in this class of spaces is obtained when $\mu=\delta_\zeta$ is a Dirac delta anchored at the point $\zeta\in\mathbb{T}$. In this case, we have the so-called \emph{local} Dirichlet space, that we denote $\mathcal{D}_\zeta$. This is of great importance in this theory because of a disintegration formula that allows to express more general Dirichlet integrals in terms of the local one, that is,
\begin{equation*}
    \mathcal{D}_\mu(f)=\int_{\mathbb{T}} \mathcal{D}_\zeta(f) \operatorname{d}\!\mu(\zeta).
    \end{equation*}
    
    In \cite{Sarason1997LocalDS}, Sarason showed that 
    $H(b)$ and $\mathcal{D}_\mu$  are not two
    completely separate classes of spaces. He proved that $\mathcal{D}_\zeta$ coincides, with equality of norms, with $H(b_\zeta)$, where $b_\zeta$ is  the rational function
\begin{equation}\label{E:b1}
    b_\zeta(z):= \frac{(1-s_0)\overline{\zeta}z}{1-s_0 \overline{\zeta}z}, \qquad z\in\mathbb D,
\end{equation}
where $s_0=(3-\sqrt{5})/2$. 
This means that $H(b_\zeta) = \mathcal{D}_\zeta$ as sets and 
\begin{equation}\label{E:equality}
    \|f\|_{H(b_\zeta)}= \|f\|_{\mathcal{D}_\zeta}, \qquad f \in H(b_\zeta).
\end{equation}

Later, Chevrot, Guillot and Ransford \cite{Chevrot2010} showed that the case studied by Sarason is basically the only case in which it holds $H(b) = \mathcal{D}_\mu$ with equality of norms as in \eqref{E:equality}. 
 In \cite{costara2013}, Costara and Ransford showed that
 there can be an equality of sets $H(b) = \mathcal{D}_\mu$ with just
 an equivalence of norms, when $\mu$ is a finitely supported measure,
 i.e., a finite sum of Dirac deltas centered at boundary points, and
 the \emph{Pythagorean mate} of $b$, that we will properly define
 later, is a rational function. 

Recently, in \cite{bellavita2023embedding}, the problem of the embedding $H(b)\hookrightarrow\mathcal{D}_\mu$ was dealt with in a greater generality, and in \cite{differencequotientoperator} a quantitative study of such embedding is carried out. In \cite{malmanseco}, the opposite problem of embeddings into the $H(b)$ space was discussed.

The main goal of this work is to sharpen the
result of Costara--Ransford in \cite{costara2013}, dropping the assumption that the Pythagorean mate of $b$ is rational. In Theorem~\ref{T:equalitymuatom}, we completely characterize the case when
$H(b)$ coincides with $\mathcal{D}_\mu$ for a finitely supported measure $\mu$. This allows to obtain an explicit description of $H(b)$ spaces for a wider class of functions $b$, using the crucial
information obtained by the equality $H(b)=\mathcal{D}_\mu$. We
also discuss a similar result for the de
Branges--Rovnyak spaces  associated to outer rational
functions. With Theorem \ref{T:mainpolynomial}, we 
describe which $H(b)$ spaces have the form 
\[H(b)=\big\{q+h\prod_{j=1}^N(z-\zeta_j)^{m_j}\colon h\in H^2,\, q\,\,\text{polynomial of degree}\,\leq \sum_{j=1}^N m_j\big\},\]
where $\zeta_1,\ldots,\zeta_N\in\mathbb{T}$ and $m_1,\ldots,m_N\geq 1$ are integers. Although these results offer a complete picture from a theoretic point of view, they may be difficult to apply since in general the Pythagorean mate of $b$ is not explicit. To overcome this, we also discuss a direct application in the case were $b$ is continuous up to the boundary, and we produce an interesting example: we show that, when $b$ is the exponential function $b(z)=e^{z^N-1}$, then $H(b)=\mathcal{D}_\mu$, $\mu$ being the measure supported on the set of the $N$-th roots of $1$. See Theorems \ref{T:applicationbcontinuous} and \ref{T:Hbexp}.

The paper is structured in the following way. In Section \ref{S:preliminaries} we gather basic definitions and important preliminary results, so that in the same section we can state the main theorems. Section \ref{S:results} is devoted to proving the main theorems, whereas Section \ref{S:examples} contains further results, examples and applications. 

\section{Preliminaries and main results} \label{S:preliminaries}
 In this work, every space of analytic function that we consider is contained in the Hardy space
\[H^2=\big\{f\in\operatorname{Hol}(\mathbb{D})\colon \|f\|_{H^2}^2:= \sup_{0<r<1} \int_{\mathbb{T}}|f(r\zeta)|^2\operatorname{d}\!m(\zeta) <\infty\big\},\]
where $m$ is the Lebesgue measure on $\mathbb{T}$, normalized so that $m(\mathbb{T})=1$.

Throughout the paper, the function $b$ is a bounded analytic function with $\|b\|_{H^\infty}=1$ that is a non-extreme point of $H^\infty_1$, the closed unit ball of $H^\infty$. This means that
\[\int_\mathbb{T}\log(1-|b|)\operatorname{d}\!m >-\infty,\]
and there exists a unique outer function $a$ satisfying $a(0)>0$ and
\[|a|^2+|b|^2=1\quad m\text{-a.e. on}\,\,\mathbb{T}.\]
 This function $a$ is defined by 
\[a(z):= \exp{\left(\int_\mathbb{T} \frac{\lambda+z}{\lambda-z}\log(1-|b(\lambda)|^2)^\frac{1}{2}\operatorname{d}\!m(\lambda)\right)}, \qquad z\in\mathbb{D}.\]
We call $a$ the (Pythagorean) mate of $b$, and we say that $(b,a)$ forms a (Pythagorean) pair. An application of the maximum principle for subharmonic functions yields that, if $(b,a)$ is a pair, then
\begin{equation} \label{E:pairinD}
    |b(z)|^2+|a(z)|^2 \leq 1 \qquad z\in\mathbb{D},
\end{equation}
see also \cite[Exercise 23.1.1]{hb2}. Given a function $b$ with $\|b\|_{H^\infty}=1$, we introduce the (boundary) spectrum of $b$, that is the set
\[\sigma(b):=\{\zeta\in\mathbb{T}\colon \liminf_{z\to\zeta}|b(z)|<1\}.\]
Notice that, by \eqref{E:pairinD}, we have that
\begin{equation} \label{E:TminussigmainZa}
    \mathbb{T}\setminus \sigma(b)=\{\zeta\in\mathbb{T}\colon \lim_{z\in\mathbb{D},z\to\zeta}|b(z)|=1\}\subseteq \{\zeta\in\mathbb{T}\colon \lim_{z\in\mathbb{D},z\to\zeta}a(z)=0\}.
\end{equation}
Notice also that, for a non-extreme function $b$, necessarily $m(\mathbb{T}\setminus\sigma(b))=0$, so that the spectrum has full Lebesgue measure.

Given $b\in H^\infty_1$, the operator
\[T_b\colon H^2\to H^2, f\mapsto bf\]
is bounded with operator norm $\|T_b\|_{H^2\to H^2}\leq 1$. In particular, the operator $I-T_bT_b^*$ is positive, and we define the de Branges--Rovnyak space $H(b)$ as the range of its square root, $H(b)=\operatorname{Ran}(I-T_bT_b^*)^\frac{1}{2}$. This is a Hilbert space with respect to the so-called \emph{range norm}, that is,
\[\|(I-T_bT_b^*)^\frac{1}{2}f\|_{H(b)}:=\|f\|_{H^2}.\]

This construction is rather abstract, there is no explicit formula that one can use to establish whether a function $f\in H^2$ belongs to a certain $H(b)$ space. This aspect is one of the features that makes this whole theory complicated, yet intriguing. An equivalent definition of these spaces is obtained within the context of the reproducing kernel Hilbert spaces: the space $H(b)$ is the RKHS associated to the kernel
\[k_w^b(z)=\frac{1-\overline{b(w)}b(z)}{1-\overline{w}z},\qquad z,w\in\mathbb{D}.\]
Again, the construction is rather abstract, as it is realized through the completion of the linear span of such kernels, but it provides explicit examples of element of $H(b)$: the kernels themselves. When the function $b$ is non-extreme, then the associated de Branges--Rovnyak space also contains all the polynomials and the function $b$ itself. In particular, $H(b)$ is a proper dense subset of $H^2$. Standard references on this subject are \cite{hb1, hb2, sarHb}.

Given a positive measure $\mu$ on $\mathbb{T}$, we recall that its
support $\operatorname{supp}(\mu)$ is the (closed) set of points
$\zeta\in\mathbb{T}$ such that every arc $\Delta$ containing $\zeta$
has positive measure $\mu(\Delta)>0$. We say that a positive measure $\mu$ on $\mathbb{T}$ is carried by a set $E\subseteq\mathbb{T}$, or that $E$ is a carrier for $\mu$, if $\mu(\mathbb{T}\setminus E)=0$.

We consider a finitely supported positive measure $\mu$ on $\mathbb{T}$, that is, a measure with $\operatorname{supp}(\mu)=\{\zeta_1,\ldots,\zeta_N\}$, with $\zeta_j\in\mathbb{T}, j=1,\ldots,N$. Without loss of generality, we can write
\begin{equation} \label{E:measurefin}
    \mu = \sum_{j=1}^N \delta_{\zeta_j},
  \end{equation} 
that is, we may assume that every atom has mass $1$, since different choices of
masses  give rise  to a weighted Dirichlet space that has the same
elements and an equivalent norm.  

Given a finite positive Borel measure $\mu$ on the unit circle $\mathbb{T}$, we define the \emph{potential}
\begin{equation} \label{E:potential}
    V_\mu \colon \mathbb{C}\ni z\mapsto \int_\mathbb{T} \frac{1}{|\lambda-z|^2}\operatorname{d}\!\mu(\lambda)\in[0,+\infty].
\end{equation}

Notice that, choosing $\mu$ as in \eqref{E:measurefin}, then
\[V_\mu(z)=\sum_{j=1}^N\frac{1}{|z-\zeta_j|^2},\qquad z\in\mathbb{C},\]
and it satisfies $V_\mu(z)=+\infty$ if and only if $ z\in\{\zeta_1,\ldots,\zeta_N\}$.

In \cite{costara2013}, the authors proved the following theorem.
\begin{theorem}[Costara--Ransford]\label{T:CostRans}
    Let $(b,a)$ be a pair such that $a$ is rational, and let $\mu$ be a finite positive measure on $\mathbb{T}$. Then $H(b)=\mathcal{D}_\mu$ if and only if the following conditions hold:
    \begin{itemize}
        \item the zeros of $a$ on $\mathbb{T}$ are all simple;
        \item the support of $\mu$ is exactly equal to this set of zeros;
        \item none of these zeros lie in the spectrum $\sigma(b_i)$, where $b_i$ is the inner factor of $b$.
    \end{itemize}
\end{theorem}
In order to comply with the assumptions of this theorem, necessarily the measure~$\mu$ is finitely supported and $\operatorname{supp}(\mu)=\{\zeta\in\mathbb{T}\colon a(\zeta)=0\}$. 
\begin{remark} As a byproduct of Theorem \ref{T:CostRans}, in \cite{costara2013} the authors showed that $\mathcal{D}_\mu = H(b_\mu)$, where the function $b_\mu$ can be chosen as the non-extreme function of $H^\infty_1$ having for Pythagorean mate $a_\mu$ a polynomial with simple zeros in the atoms of~$\mu$. We can write 
 \begin{equation} \label{D:arational}
     a_\mu(z) = C\prod_{j=1}^N(z-\zeta_j), \qquad z\in\mathbb{D},
 \end{equation}
 where $C\in\mathbb{C}$ is an appropriate constant such that $a_\mu(0)>0$ and $\|a_\mu\|_{H^\infty}\leq 1.$ Also, in \cite{costara2013} the function $b_\mu$ is chosen to be a polynomial whose zeros lie outside the  disk~$\mathbb{D}$, as it is constructed starting from the function $a_\mu$ defined in \eqref{D:arational} using the classical Fejér-Riesz Theorem: given $\omega(e^{i\theta})=\sum_{j=-n}^n c_j e^{ij\theta}$ a non-zero trigonometric polynomial that assumes non-negative values for all $\theta\in[0,2\pi]$, there exists an analytic polynomial $p(z)=\sum_{j=0}^n a_jz^j$ with no zeros on $\mathbb{D}$ such that $\omega(e^{i\theta})=|p(e^{i\theta})|^2$ for every $\theta$. For a proof, see \cite[Theorem 27.19]{hb2}. In this context, $\omega=1-|a_\mu|^2$.
\end{remark}

From now on, we fix the notation $(b_\mu,a_\mu)$ for this special pair of polynomials that we just described. The choice of the pair $(b,a)$ that realizes $H(b)=\mathcal{D}_\mu$ is not unique. For example, in \cite{Sarason1997LocalDS} the equality $H(b)=\mathcal{D}_\zeta$ is realized with the pair $(b_\zeta,a_\zeta)$ given by
 \[b_\zeta(z) = \frac{(1-s_0)\overline{\zeta} z}{1-s_0\overline{\zeta} z}, \qquad a_\zeta(z) = \frac{(1-s_0)(1-\overline{\zeta}z)}{1-s_0\overline{\zeta}z}, \qquad z\in\mathbb{D},\]
 where 
 \[s_0=\frac{3-\sqrt{5}}{2}.\]
 
 In \cite{chacon2013}, it is shown that if $\mu$ is chosen as in \eqref{E:measurefin}, the space $\mathcal{D}_\mu$ has the explicit presentation
 \begin{equation}\label{E:Dmuexplicit}
     \mathcal{D}_\mu = \{a_\mu h + p \colon h\in H^2, p\,\,\text{polynomial of degree}\,\, N-1\}.
 \end{equation}
 
 In particular, proving that for a given $b$ we have that $H(b)=\mathcal{D}_\mu$ with such~$\mu$, we obtain an explicit description of the elements of $H(b)$. 
 
 We are ready to state our main original results. We provide two results concerning the separate embeddings $H(b)\hookrightarrow \mathcal{D}_\mu$ and $\mathcal{D}_\mu \hookrightarrow H(b)$, with
 $\mu$ finitely supported, and finally we deal with the full equality of sets. Here, by embedding we mean that we have a set inclusion, and that the inclusion map is bounded. Notice that, by the closed graph theorem, the set inclusion automatically ensures the boundedness
 required. 
 
\begin{theorem2} \label{T:inclusionmuatom}
     Let $b$ be a non-extreme function in $H^\infty_1$, $a$ its Pythagorean mate, and $\mu$ an atomic measure as in \eqref{E:measurefin}. The following are equivalent.
     \begin{enumerate}[label={\rm (\Alph*)}]
         \item \label{T:inclusionmuatom1} The embedding $H(b) \hookrightarrow \mathcal{D}_\mu$ holds.
         \item \label{T:inclusionmuatom2} $\{\zeta_1,\ldots,\zeta_N\} \subseteq \mathbb{T} \setminus \sigma(b)$ and $\sup_{\mathbb{D}} |a|^2 V_\mu < \infty$.
          \item \label{T:inclusionmuatom4} $\{ \zeta_1, \ldots, \zeta_N \}\subseteq \mathbb{T} \setminus \sigma(b)$ and there exists $g\in H^\infty$ such that $a$ has the form 
         \begin{equation*}
             a=\left(\prod_{j=1}^N (z-\zeta_j)\right) g.
         \end{equation*}
     \end{enumerate}
 \end{theorem2}

Notice that, necessarily, the function $a$ has to vanish on all the points $\zeta_1,\ldots,\zeta_N$, to compensate for $V_\mu$. We move to the reverse inclusion $\mathcal{D}_\mu \hookrightarrow H(b)$.  
 \begin{theorem2} \label{T:rev-emb}
     Let $b$ be a non-extreme function in $H^\infty_1$, $a$ its Pythagorean mate, and $\mu$ an atomic measure as in \eqref{E:measurefin}. The following are equivalent.
     \begin{enumerate}[label={\rm (\Alph*)}]
         \item \label{T:rev-emb1} The embedding $\mathcal{D}_\mu\hookrightarrow H(b)$ holds.
          \item \label{T:rev-emb4} There exists $g\in \operatorname{Hol}(\mathbb{D})$ with $\inf_\mathbb{D}|g|>0$ such that $a$ has the form 
         \begin{equation*}
             a=\left(\prod_{j=1}^N (z-\zeta_j)\right) g.
         \end{equation*}
         \item \label{T:rev-emb2}  $\inf_{\mathbb{D}} |a|^2 V_\mu >0$.
     \end{enumerate}
 \end{theorem2}

Notice that, in this other case, the function $a$ cannot vanish outside of the set $\{ \zeta_1, \ldots, \zeta_N\}$, since $V_\mu$ is always strictly positive. Also, if we  multiply the function~$b$ for an inner factor $u$, that is, if we consider $\tilde{b}:=bu$, then $(\tilde{b},a)$ still forms a Pythagorean pair, and $\tilde{b}$ satisfies the assumptions of Theorem \ref{T:rev-emb} if and only if $b$ does. This is very natural, in light of the decomposition \cite[Theorem 18.7]{hb2}
\[
  H(bu)=H(b)+bH(u) \supseteq H(b).
\] As for the first embedding result, Theorem $\ref{T:inclusionmuatom}$, notice that the condition $\{\zeta_1,\ldots,\zeta_N\}\subseteq\mathbb{T}\setminus\sigma(b)$ is not stable under multiplication of $b$ by inner functions. 

\begin{theorem2} \label{T:equalitymuatom}
     Let $b$ be a non-extreme function in $H^\infty_1$, $a$ its Pythagorean mate, and $\mu$ an atomic measure as in \eqref{E:measurefin}. The following are equivalent.
      \begin{enumerate}[label={\rm (\Alph*)}]
         \item \label{T:eqmuatom1} The equality $H(b) = \mathcal{D}_\mu$ holds.
         \item \label{T:eqmuatom2} $\{\zeta_1,\ldots,\zeta_N\}= \mathbb{T} \setminus \sigma(b)$ and $0<\inf_{\mathbb{D}}|a|^2 V_\mu\leq\sup_{\mathbb{D}} |a|^2 V_\mu < \infty$.
          \item \label{T:eqmuatom4} $\{ \zeta_1, \ldots, \zeta_N \}= \mathbb{T} \setminus \sigma(b)$ and there exists $g\in H^\infty$ with $\inf_\mathbb{D}|g|>0$ such that $a$ has the form 
         \begin{equation*}
             a=\left(\prod_{j=1}^N (z-\zeta_j)\right) g.
         \end{equation*}
     \end{enumerate}
 \end{theorem2}

The explicit expression of $a$ in \ref{T:eqmuatom4} of Theorem \ref{T:equalitymuatom} generalizes the result of Costara--Ransford. For a general $a\in H^\infty$ that is not necessarily a rational function, a priori it makes no sense to talk about the order of zeros in boundary points $\zeta\in\mathbb{T}$. However, saying that 
\[a(z)=g(z)\prod_{j=1}^N(z-\zeta_j), \qquad z\in\mathbb{D},\]
with $g\in H^\infty$ and $\inf_\mathbb{D}|g|>0$, clearly resembles the condition that $a$ has zeros of order one in the points $\zeta_1,\ldots,\zeta_N$. 

Finally, we extend Theorem \ref{T:equalitymuatom}, allowing a higher multiplicity for the zeros of~$a$. In this case, we will no longer obtain a harmonically weighted Dirichlet space. We now describe what will be the model for such de Branges--Rovnyak spaces. Together with the points $\zeta_1,\ldots,\zeta_N\in\mathbb{T}$, we fix integers $m_1,\ldots,m_N \geq 1$. Let 
\[p_a(z)=C\prod_{j=1}^N(z-\zeta_j)^{m_j},\]
with $C\in\mathbb{C}$ such that $p_a(0)>0$ and $\|p_a\|_{H^\infty}=1$. Now, by the Fejer-Riesz Theorem, let $p_b$ be an analytic polynomial such that $1-|p_a|^2=|p_b|^2$ on $\mathbb{T}$. In particular, $\|p_b\|_{H^\infty}=1$ and 
\[\sigma(p_b)=\mathbb{T}\setminus \{\zeta_1,\ldots,\zeta_N\}.\]
Also, by \cite{concrete}, the de Branges--Rovnyak space $H(p_b)$ has the explicit description
\[H(p_b)=\{p_ah+ q\colon h\in H^2,\, q\,\,\text{polynomial of degree}\,\leq M\},\]
$M$ being the sum $\sum_{j=1}^N m_j$. 

\begin{theorem2} \label{T:mainpolynomial}
    Let $(b,a)$ be a pair and $(p_b,p_a)$ as above. Then, $H(b)=H(p_b)$ if and only if $\sigma(b)=\mathbb{T}\setminus \{\zeta_1,\ldots,\zeta_N\}$ and there exists a function $g\in H^\infty$ such that $\inf_{\mathbb{D}}|g|>0$ and $a=p_a g$.
\end{theorem2}

Whenever $(b,a)$ is a pair that satisfies the assumption of the previous theorem, that is, there exist points $\zeta_1,\ldots,\zeta_N\in\mathbb{T}$ and positive natural numbers $m_1,\ldots,m_N$ such that $\{\zeta_1,\ldots,\zeta_N\}=\mathbb{T}\setminus\sigma(b)$ and $a=p_ag$, for $g\in H^\infty$ with $\inf_\mathbb{D}|g|>0$, we say that $(b,a)$ is a pair of polynomial-type.

\section{Proofs of the main results} \label{S:results}
We begin by listing some properties of the potential $V_\mu$ introduced in \eqref{E:potential} that we will use in this work. 
    \begin{itemize}
        \item $V_\mu$ is lower-semicontinuous on $\mathbb{C}$ and continuous on $\mathbb{C}\setminus\text{supp}(\mu)$;
        \item For $z\in\mathbb{C}$, it holds
        \begin{equation*}
            \frac{\mu(\mathbb{T})}{(1+|z|)^2} \leq V_\mu(z) \leq \frac{\mu(\mathbb{T})}{\operatorname{dist}(z,\text{supp}(\mu))^2};
        \end{equation*}
        \item $V_\mu = \infty$ $\mu$-a.e. on $\mathbb{T}$.
    \end{itemize}

The following lemma is a reformulation of Lemmas 3.4 and 3.5 of \cite{costara2013}.
\begin{lemma} \label{L:normcw}
    Let 
    \[c_w(z)= \frac{1}{1-\overline{w}z}, \qquad w,z\in\mathbb{D},\]
    be the Cauchy--Szeg\"o kernel. Then, if $(b,a)$ is a pair,
    \begin{equation}
        \|c_w\|_{H(b)}^2 = \frac{1+|b(w)/a(w)|^2}{1-|w|^2},
    \end{equation}
    and if $\mu$ is a finite positive Borel measure on $\mathbb T$, then
    \begin{equation}
        \|c_w\|_{\mathcal{D}_\mu}^2 =\frac{1+|w|^2V_\mu(w)}{1-|w|^2}.
    \end{equation}
\end{lemma}

 Being that $\mathcal{D}_\mu = H(b_\mu)$, to study the equalities $\mathcal{D}_\mu = H(b)$ for a general $b$ we can appeal to the  following result from Ball and Kriete \cite{ballkriete}; see also \cite[Theorem 27.15]{hb2}.
 
 \begin{theorem} \label{T:embeddingBK}
     Let $b_2, b_1$ be functions in the closed unit ball of $H^\infty$, with $b_1$ non-extreme. Then, $H(b_2) \subseteq H(b_1)$ if and only if the two following conditions hold:
         \begin{enumerate}[label={\rm (\roman*)}]
             \item \label{E:embeddingBKa} There exist $v,w\in H^\infty$ 
 such that $b_1 +v a_1 = b_2 w$;
             \item \label{E:embeddingBKb} There exists $\gamma>0$ such that $1-|b_2|^2\leq \gamma(1-|b_1|^2)$ a.e. on $\mathbb{T}$.
         \end{enumerate}
 \end{theorem}

To verify condition \ref{E:embeddingBKa}, we will often invoke the
Corona theorem. The result that we will use is the following: given
two functions $f_1,f_2\in H^\infty$ such that
$\inf_{\mathbb{D}}(|f_1|+|f_2|)>0$, there exist two functions
$g_1,g_2\in H^\infty$ such that $f_1g_1+f_2g_2=1$. We will say that
$f_1,f_2$ form a Corona pair. For further reference and the full
statement of the Corona theorem, see \cite[Theorem 2.1]{garnett}.

The following result is a generalization of Theorem 1.4 in \cite{bellavita2023embedding}, and it is a necessary condition for the set inclusion $H(b)\subseteq\mathcal{D}_\mu$, which is of interest in its own right. Notice that this applies to general measures and functions $b$ not necessarily non-extreme.

\begin{theorem}\label{T:HbinDmnec}
Let $\mu$ be a finite positive Borel measure on $\mathbb{T}$ and let $b\in H^\infty_1$. If the embedding $H(b) \hookrightarrow \mathcal{D}_\mu$ holds, then $\mu\big(\sigma(b)\big)=0$.
\end{theorem} 

\begin{proof} 
For the proof, we recall the potential $V_\mu\colon\mathbb{C}\to[0,+\infty]$ defined in \eqref{E:potential}. We show that $V_\mu$ is finite on the boundary spectrum $\sigma(b)$, which we can assume to be non-empty without loss of generality. Let $C>0$ be a constant such that
\[\mathcal{D}_\mu(f) \leq C\|f\|_{H(b)}^2, \qquad f\in H(b).\]
Let $\lambda\in\sigma(b)$ and let us consider a sequence $(w_n)_n$ in $\mathbb{D}$ such that $|b(w_n)| \to \beta\in [0,1)$ as $n\to\infty$. As it was done in \cite{bellavita2023embedding}, we have that
\begin{align*}
    C\|k_{w_n}^b\|_{H(b)}^2&\geq    \mathcal{D}_\mu(k_{w_n}^b)=  \int_{\mathbb{T}} \mathcal{D}_\zeta(k_{w_n}^b) \,  \operatorname{d}\!\mu(\zeta) \\
    &\geq   \int_{\mathbb{T}} \|k_{w_n}^b\|_{H(b)}^2 \frac{\big(1-|b(w_n)|\big)^2 \big(|w_n| - |b(w_n)|\big)^2}{|\zeta-w_n|^2 \big(1-|b(w_n)|^2\big)} \, \operatorname{d}\!\mu(\zeta) \\
    &=   \|k_{w_n}^b\|^2_{H(b)} \frac{\big(1-|b(w_n)|\big)\big(|w_n| - |b(w_n)|\big)^2}{1+|b(w_n)|} \int_{\mathbb{T}} \frac{1}{|\zeta -w_n|^2}  \, \operatorname{d}\!\mu(\zeta) .
\end{align*}
Hence, by Fatou's Lemma, it holds that 
\begin{align*}
    C &\geq \liminf_{n} \frac{\big(1-|b(w_n)|\big)\big(|w_n| - |b(w_n)|\big)^2}{1+|b(w_n)|} \int_{\mathbb{T}} \frac{1}{|\zeta -w_n|^2}  \, \operatorname{d}\!\mu(\zeta) \\
    &\geq \frac{(1-\beta)^3}{1+\beta} \int_{\mathbb{T}} \frac{1}{|\zeta -\lambda|^2}  \, \operatorname{d}\!\mu(\zeta) \\
    &= \frac{(1-\beta)^3}{1+\beta} V_\mu(\lambda).
\end{align*}
In particular,
\[V_\mu(\lambda) \leq C \frac{1+\beta}{(1-\beta)^3} <\infty, \qquad \lambda\in\sigma(b).\]
The theorem follows from the fact that $V_\mu = \infty$ $\mu$-a.e. on $\mathbb{T}$ and therefore, necessarily, we have that $\mu\big(\sigma(b)\big)=0$.
\end{proof}

    We remark that, if there exists an $\varepsilon>0$ such that for every $\lambda\in\sigma(b)$ we have $\beta<1-\varepsilon$, then the potential $V_\mu$ is uniformly bounded on $\sigma(b)$. This happens, for example, when $b$ is inner. In this case, we can always choose $\beta=0$. In general, however, $V_\mu$ can be unbounded on $\sigma(b)$. As an example, take the usual function~$b_1$ as in \eqref{E:b1} that realizes the equality $H(b_1)=\mathcal{D}_1$. Clearly, the embedding holds, but the potential associated to the measure $\mu=\delta_1$ is
    \[V_\mu(z) = \frac{1}{|z-1|^2},\]
    which is not bounded on $\sigma(b)=\mathbb{T}\setminus\{1\}.$  

 We are now ready to prove the main theorems.

 \begin{proof}[Proof of Theorem \ref{T:inclusionmuatom}]
     We show that \ref{T:inclusionmuatom1} implies \ref{T:inclusionmuatom2}. Since
     \[\mathcal{D}_\mu= \bigcap_{j=1}^N\mathcal{D}_{\zeta_j},\] 
     then by Theorem \ref{T:HbinDmnec},  necessarily $\zeta_j\notin\sigma(b)$ for each $j=1,\ldots,N$, proving that $\{\zeta_1,\ldots,\zeta_N\} \subseteq \mathbb{T} \setminus \sigma(b)$.
     Now, by the boundedness of the embedding and from Lemma \ref{L:normcw}, it follows that
     \[1+|w|^2V_\mu(w)\lesssim \frac{|a(w)|^2+|b(w)|^2}{|a(w)|^2}, \qquad w\in\mathbb{D}.\]
     In particular, $\sup_{\mathbb{D}} |a|^2 V_\mu < \infty$. To see this, we consider two cases. If $|w|\geq 1/2,$ then 
     \[|a(w)|^2V_\mu(w)\leq 4|a(w)|^2(1+|w|^2V_\mu(w))\lesssim |a(w)|^2+|b(w)|^2 \leq 2. \]
     
     If $|w|\leq 1/2,$ then
     \[|a(w)|^2V_\mu(w)\leq \int_{\mathbb{T}}\frac{1}{|w-\zeta|^2}\operatorname{d}\!\mu(\zeta) \leq \int_{\mathbb{T}}\frac{1}{(1-|w|)^2}\operatorname{d}\!\mu(\zeta)\leq 4\mu(\mathbb{T}). \]

The implication \ref{T:inclusionmuatom2} $\implies$ \ref{T:inclusionmuatom4} is trivial, as the function
     \[g(z):=\frac{a(z)}{\prod_{j=1}^N(z - \zeta_j)}, \qquad z\in\mathbb{D},\]
     belongs to $H^\infty$ by assumption.
     
     The last implication \ref{T:inclusionmuatom4} $\implies$ \ref{T:inclusionmuatom1} is more convoluted, and makes use of Theorem~\ref{T:embeddingBK}. We recall that $\mathcal{D}_\mu = H(b_\mu)$, where the pair $(b_\mu,a_\mu)$ is described in \eqref{D:arational}. Condition \ref{E:embeddingBKb} of Theorem \ref{T:embeddingBK} is
     $1-|b|^2 \leq \gamma(1-|b_\mu|^2),$ a.e. on $\mathbb{T}$, and this follows by assumption, since $a=a_\mu g$ with $g\in H^\infty$. On the other hand, we also have that $\{\zeta_1,\ldots,\zeta_N\}\subseteq \mathbb{T}\setminus\sigma(b),$ meaning that 
     \[\lim_{z\in\mathbb{D},z\to\zeta_j} |b(z)|=1\]
     for every $j=1,\ldots,N$. By the definition \eqref{D:arational} of $a_\mu$, it follows that $b$ and $a_\mu$ form a Corona pair. An application of the Corona Theorem proves that condition \ref{E:embeddingBKa} holds. We conclude that $H(b) \subseteq H(b_\mu) = \mathcal{D}_\mu.$
 \end{proof}

Notice that, in particular, when $H(b)\hookrightarrow\mathcal{D}_\mu$ with $\mu$ as in \eqref{E:measurefin},
\[\{\zeta_1,\ldots,\zeta_N\}\subseteq\{\zeta\in\mathbb{T}\colon \lim_{z\to\zeta}|a(z)|=0\}.\]

As an application of this theorem to the special case of one atom, $N=1$, we obtain a complete characterization for the embedding in the local Dirichlet space $H(b) \hookrightarrow \mathcal{D}_\zeta$, completing the analysis done in \cite{bellavita2023embedding}, for a non-extreme $b$.

\begin{corollary}
    Let $b$ be a non-extreme function in $H^\infty_1$, $a$ its Pythagorean mate, and $\zeta\in\mathbb{T}$. Then, it holds the embedding $H(b)\hookrightarrow\mathcal{D}_\zeta$ if and only if  $\zeta\notin\sigma(b)$ and there exists $g\in H^\infty$ such that $a=(z-\zeta)g$. 
\end{corollary}

Next, we prove Theorem \ref{T:rev-emb}.

 \begin{proof}[Proof of Theorem \ref{T:rev-emb}]
     We show that \ref{T:rev-emb1} implies \ref{T:rev-emb4}. If $\mathcal{D}_\mu \subseteq \ H(b),$ then in particular by \eqref{E:Dmuexplicit} we have that $a_\mu H^2\subseteq H(b)$. By \cite[Theorem 23.6]{hb2}, this implies that the quotient $h:=a_\mu/a\in H^\infty$. Writing $a_\mu=ah$, since both $a,a_\mu$ are outer functions, we have that $h$ is outer as well. In particular, setting $g:= 1/h$, we obtain \ref{T:rev-emb4}. 

      The implication \ref{T:rev-emb4} $\implies$ \ref{T:rev-emb2} is straightforward. 
      
      To conclude, we show that \ref{T:rev-emb2} implies \ref{T:rev-emb1}. We assume that $\inf_{\mathbb{D}} |a|^2 V_\mu >0$, and using the identity $\mathcal{D}_\mu=H(b_\mu)$ we apply Theorem \ref{T:embeddingBK} to prove that $H(b_\mu)\subseteq H(b).$ Condition \ref{E:embeddingBKb} follows at once from the assumption. To prove condition \ref{E:embeddingBKa}, we show that $(b_\mu,a)$ forms a Corona pair. By assumption,
      \[\inf_{z\in\mathbb{D}} |a(z)|^2V_\mu(z) =\inf_{z\in\mathbb{D}}\sum_{j=1}^N \frac{|a(z)|^2}{|z-\zeta_j|^2}>0.\]
      In particular, for every $\lambda\in\mathbb{T}\setminus\{\zeta_1,\ldots,\zeta_N\}$, we have that $\liminf_{z\to\lambda}|a(z)|>0.$ On the other hand, since $(b_\mu,a_\mu)$ is a pair of two polynomials, the relation $|b|^2+|a|^2=1$ holds on all the unit circle $\mathbb{T}$, and in particular for every $j=1,\ldots,N$ we have that $\lim_{z\to\zeta_j} |b_\mu(z)|=1.$ This shows that $\inf_{\mathbb{D}}(|a_\mu|+|b|)>0$, concluding the proof.
 \end{proof}

Notice that, in particular, when $\mathcal{D}_\mu\hookrightarrow H(b)$ with $\mu$ as in \eqref{E:measurefin},
\begin{equation} \label{E:setinequality}
\{ \zeta\in\mathbb{T} \colon \liminf_{z\to\zeta} |a(z)|=0\}\subseteq\{\zeta_1,\ldots,\zeta_N\}.    
\end{equation}

We are now in a position to prove the full equality result.

\begin{proof}[Proof of Theorem \ref{T:equalitymuatom}]
    We assume that $H(b) = \mathcal{D}_\mu$, and we show that $\ref{T:eqmuatom2}$ holds. By Theorems \ref{T:inclusionmuatom} and \ref{T:rev-emb}, we only have to show that $\mathbb{T}\setminus\sigma(b)\subseteq\{\zeta_1,\ldots,\zeta_N\}$. This follows from the fact that
    \[\mathbb{T}\setminus\sigma(b)\subseteq\{\zeta\in\mathbb{T}\colon \lim_{\mathbb{D}\ni z\to\zeta}a(z)=0\},\]
    for $|a|^2+|b|^2\leq 1$ on $\mathbb{D}$, and \eqref{E:setinequality}.
    The rest of the proof follows at once from Theorems \ref{T:inclusionmuatom} and \ref{T:rev-emb}.
\end{proof}

Notice that, in particular, when $H(b)=\mathcal{D}_\mu$ with $\mu$ as in \eqref{E:measurefin}, then
\begin{equation*}  
\{\zeta_1,\ldots,\zeta_N\}=\{\zeta\in\mathbb{T} \colon \liminf_{z\to\zeta} |a(z)|=0\}=\{\zeta\in\mathbb{T} \colon \lim_{z\to\zeta} |a(z)|=0\}=\mathbb{T}\setminus\sigma(b).
\end{equation*}

Before proving Theorem \ref{T:mainpolynomial}, we recall the Smirnov class, defined as
    \[\mathcal{N}_+:=\big\{h_1/h_2\colon h_1\in H^\infty, h_2\in H^\infty\setminus\{0\}\,\,\text{outer}\,\big\}.\]

Such functions have the following important property, often referred to as the Smirnov maximum principle: if $f\in\mathcal{N}_+$ and 
 \[\int_\mathbb{T} |f|^2 \operatorname{d}\!m <\infty,\]
 then $f\in H^2$. Also, if the boundary function $f$ belongs to $L^\infty$, then $f\in H^\infty$.

\begin{proof}[Proof of Theorem \ref{T:mainpolynomial}]
    We start assuming that $H(b)=H(p_b)$. By condition \ref{E:embeddingBKb} of Theorem \ref{T:embeddingBK}, we have that both $a/p_a$ and $p_a/a$ belong to $L^\infty$, and then by Smirnov's maximum principle we conclude that there exists a function $g\in H^\infty$ such that $\inf_{\mathbb{D}}|g|>0$ and $a=p_a g$ on $\mathbb{D}$. By \eqref{E:TminussigmainZa}, this also shows that 
    \[\mathbb{T}\setminus\sigma(b)\subseteq \{\zeta_1,\ldots,\zeta_N\}.\]
    Finally, by Theorem \ref{T:inclusionmuatom}, we have that $H(b)=H(p_b)\subseteq \mathcal{D}_\mu$, $\mu$ being the positive measure
    \[\mu=\sum_{j=1}^N\delta_{\zeta_j},\]
    and then by Theorem \ref{T:HbinDmnec} we have that  \[ \{\zeta_1,\ldots,\zeta_N\}\subseteq\mathbb{T}\setminus\sigma(b).\]
    To deal with the reverse implication, we apply Theorem \ref{T:embeddingBK}. Condition \ref{E:embeddingBKb} is clear for both the inclusions $H(b)\subseteq H(p_b)$ and $H(p_b)\subseteq H(b)$. To prove that $H(b)\subseteq H(p_b)$, it is enough to show that $b,p_a$ form a Corona pair, and this follows from the formula of $p_a$ and the assumption that $\sigma(b)=\mathbb{T}\setminus \{\zeta_1,\ldots,\zeta_N\}$: whenever $a$ vanishes, the modulus of $b$ is $1$. To prove that $H(p_b)\subseteq H(b)$, it is enough to show that $p_b,a$ form a Corona pair, and this follows from the assumption on $a$ and the fact that $|p_b(\zeta_j)|=1$ for every $j=1,\ldots,N$. We conclude that $H(b)=H(p_b)$.
\end{proof}

\section{Examples and further results} \label{S:examples}

Under the assumptions of Theorem \ref{T:mainpolynomial}, when $(b,a)$ forms a pair of polynomial-type, we also have the following information on the Clark measures $\mu_\lambda$ associated to $b$, for $\lambda\in\mathbb{T}$. We recall that the Clark measure $\mu_\lambda$ is the unique finite positive measure on $\mathbb{T}$ such that 
\[\frac{1-|b(z)|^2}{|\lambda-b(z)|^2} = \int_{\mathbb{T}}\frac{1-|z|^2}{|z-\zeta|^2}\operatorname{d}\!\mu_\lambda(\zeta), \qquad z\in\mathbb{D}.\]
See \cite[Section 9]{cauchytransform} for details. The measure $\mu_\lambda$ has the Radon-Nykodym decomposition
\[\operatorname{d}\!\mu_\lambda = \frac{1-|b|^2}{|\lambda-b|^2}\operatorname{d}\!m + \operatorname{d}\!\sigma_{\lambda},\]
and the singular part $\sigma_\lambda$ is carried by the set
\begin{equation}\label{E:carrier}
    E_\lambda=\{\eta\in\mathbb{T}\colon \lim_{r\to 1}b(r\eta)=\lambda\}.
\end{equation}

\begin{proposition}
Let $(b,a)$ be a pair of polynomial-type with $\sigma(b)=\mathbb{T}\setminus\{\zeta_1,\ldots,\zeta_N\}$ and $\lambda\in\mathbb{T}$. Then, the Clark measure of $b$ associated to $\lambda$, $\mu_\lambda$, is absolutely continuous if and only if $\lambda\notin\{b(\zeta_1),\ldots,b(\zeta_N)\},$ where the value $b(\zeta_j)$ is well-defined and intended as the non-tangential limit for each $j=1,\ldots,N$.
\end{proposition}

\begin{proof}
    Since $H(b)\subseteq\mathcal{D}_\mu$, with $\mu=\sum_{j=1}^N\delta_{\zeta_j}$, we have that every function $f\in H(b)$ admits non-tangential limit at $\zeta_j$ for every $j=1,\ldots,N$, by \cite[Theorem 7.2.1]{primer}. This is equivalent to saying that the function $b$ admits angular derivative in the sense of Caratheodory in every $\zeta_j, j=1,\ldots,N$ (see \cite[Theorem 21.1]{hb2}, and in particular the non-tangential limits
\[b(\zeta_j)=\angle\lim_{z\to\zeta_j} b(z)\]
exist and belong to $\mathbb{T}$ for every $j=1,\ldots,N$. It follows that the Clark measure~$\mu_{b(\zeta_j)}$ has an atom at $\zeta_j$, for every $j=1,\ldots,N$, by \cite[Theorem 21.5]{hb2}. In particular, for every $j=1,\ldots,N$, the measure $\mu_{b(\zeta_j)}$ is not absolutely continuous. Now, consider $\lambda\notin\{ b(\zeta_1), \ldots, b(\zeta_N) \}$. The key point is to show that the set $E_\lambda$ defined in \eqref{E:carrier}, that is a carrier for the singular part of $\mu_\lambda$, is empty. We prove that the value $\lambda$ is never attained as a radial limit by $b$, and thus the measure~$\mu_\lambda$ is absolutely continuous: we show that for every $\eta\in\mathbb{T}$ fixed we have that $\inf_{0<r<1}|\lambda-b(r\eta)|>0$. Let us argue by contradiction, and let us consider a point $\eta\in\mathbb{T}$ and a sequence $(r_n)_n$ in $(0,1)$ such that $|\lambda-b(r_n\eta)|\leq 1/n$ for every $n$. By the inequality $|a(z)|^2\leq1-|b(z)|^2,z\in\mathbb{D}$,  necessarily $\lim_{n}|a(r_n\eta)|=0$. By the assumptions on $a$, we must have that $\eta=\zeta_j$ for some $j=1,\ldots,N$, but then $\lambda = \lim_n b(r_n \eta) = b(\zeta_j)$, producing a contradiction. The result is now proved.
\end{proof}

Finally, we move to the applications of the main theorems. We recall that, when $H(b) \hookrightarrow \mathcal{D}_\mu$, then for every $j=1,\ldots,N$ the limit $\lim_{z\to\zeta_j}a(z)$ exists and it is $0$. Then, the quantity 
\[|a(z)|^2V_\mu(z)=\sum_{j=1}^N \bigg|\frac{a(z)-a(\zeta_j)}{z-\zeta_j}\bigg|^2\]
is related to the difference quotient of $a$ in the points $\zeta_1,\ldots,\zeta_N$. This is another interpretation of the interplay between the vanishing of $a$ in the points $\zeta_1,\ldots,\zeta_N$ and the fact that $V_\mu(\zeta_j)=+\infty$ for each $j=1,\ldots,N$.

\begin{theorem}
    Let $b\in H^\infty_1$ non-extreme with $\{\zeta_1,\ldots,\zeta_N\}\subseteq \mathbb{T}\setminus\sigma(b)$. Then, if for every $j=1,\ldots,N$ the limit
    \[\lim_{\mathbb{D}\ni z\to\zeta_j}\frac{a(z)-a(\zeta_j)}{z-\zeta_j}\]
    exists, we have that $H(b)\hookrightarrow\mathcal{D}_\mu$. Also, under the same assumptions, we have that  $H(b)=\mathcal{D}_\mu$ if and only if $|a'(\zeta_j)|>0$ for every $j=1,\ldots,N$.
\end{theorem}
\begin{proof}
    The proof is a trivial consequence of Theorems \ref{T:inclusionmuatom}, \ref{T:rev-emb}, \ref{T:equalitymuatom}.
\end{proof}
We provide a more concrete application of Theorem \ref{T:equalitymuatom}. In general, $b$ is given, whereas $a$ is not explicit. However, all of the results that we mentioned rely on the function $a$. Under reasonable assumptions on the function $b$, we are able to produce a theorem that no longer revolves around $a$.

\begin{theorem} \label{T:applicationbcontinuous}
    Let $b\in H^\infty_1\cap C(\overline{\mathbb{D}})$ be a non-extreme function that is continuous up to $\mathbb{T}$. Then, $H(b)=\mathcal{D}_\mu$ if and only if $\mathbb{T}\setminus\sigma(b)=\{\zeta_1,\ldots,\zeta_N\}$ and 
    \[0<\inf_{\lambda\in\sigma(b)}\frac{1-|b(\lambda)|^2}{\prod_{j=1}^N|\lambda-\zeta_j|^2}\leq \sup_{\lambda\in\sigma(b)}\frac{1-|b(\lambda)|^2}{\prod_{j=1}^N|\lambda-\zeta_j|^2} <\infty.\]
\end{theorem}

\begin{proof}
The key point is that, under continuity assumptions for the function $b$, we have clearer information on the Pythagorean mate $a$ and the relation $|a|^2+|b|^2=1$ $m$-a.e. on $\mathbb{T}$ strengthens. First of all, notice that
\[\sigma(b)=\{\zeta\in\mathbb{T}\colon |b(\zeta)|<1\}.\]
By the definition of $a$ we have that
\[|a(z)|=\exp{\left(\int_\mathbb{T} \frac{1-|z|^2}{|\lambda-z|^2}\log(1-|b(\lambda)|^2)^\frac{1}{2}\operatorname{d}\!m(\lambda)\right)},\qquad z\in\mathbb{D}.\]
In particular, since the boundary function $\log(1-|b(\lambda)|^2)^\frac{1}{2}$ is continuous at every $\lambda\in\sigma(b)$, by the properties of the Poisson kernel\footnote{See for example \cite[Theorem 1.2]{Duren}.} we have that the limit $\lim_{r\to 1}|a(r\lambda)|$ exists for every $\lambda\in\sigma(b)$ and it is equal to $\sqrt{1-|b(\lambda)|^2}$. Then, we have that 
\[|a(\lambda)|^2+|b(\lambda)|^2=1,\qquad\lambda\in\sigma(b),\]
and we recall that for a non-extreme function $b$ the spectrum $\sigma(b)$ has full Lebesgue measure. Also, since by \eqref{E:TminussigmainZa} 
 \begin{equation*} 
    \mathbb{T}\setminus \sigma(b) \subseteq \{\zeta\in\mathbb{T}\colon \lim_{z\in\mathbb{D},z\to\zeta}a(z)=0\},
\end{equation*}
we conclude that $|a|=\sqrt{1-|b|^2}$ everywhere on $\mathbb{T}$. In particular, it vanishes exactly on the set $\mathbb{T}\setminus\sigma(b)=\{\zeta_1,\ldots,\zeta_N\}.$

With this in mind, an application of Theorem \ref{T:equalitymuatom} concludes the proof: to check whether $a/p_\mu$ and $p_\mu/a$ belong to $H^\infty$, by the Smirnov maximum principle it suffices to check whether these quotients of outer functions belong to $L^\infty$, and it is enough to do so on $\sigma(b)$. 
\end{proof}

Notice that similar results can be easily obtained to study the embeddings $H(b)\hookrightarrow\mathcal{D}_\mu$ and $\mathcal{D}_\mu\hookrightarrow H(b)$, or even the equality $H(b)=H(p_b)$, when we allow a higher multiplicity for the zeros of $a$. Also, the hypothesis that $b$ is continuous up to the boundary can be weakened: what we really use is that $|b|$ is continuous on every $\lambda\in\sigma(b)$. Also, in general, multiplying $b$ by inner functions whose spectrum is disjoint from $\{\zeta_1,\ldots,\zeta_n\}$ does not disturb the identity $H(b)=\mathcal{D}_\mu$. 

We conclude with the example that was mentioned in the Introduction: when $b$ is the exponential function $e^{z^{N}-1}$, the associated $H(b)$ space is a Dirichlet space.

\begin{theorem} \label{T:Hbexp}
    Let $b(z)=e^{z^N-1}$, $\{\xi_1,\ldots,\xi_N\}$ be the set of the $N$-th roots of $1$ and $\mu$ the measure
    $\mu=\sum_{j=1}^N \delta_{\xi_j}.$ Then $H(b)$ coincides with the Dirichlet space $\mathcal{D}_\mu$.
\end{theorem}
Before the proof, we introduce some notation: given two positive functions $f,g$ on $\mathbb{T}$ and $E\subseteq\mathbb{T}$, we write $f\asymp g$ on $E$ to say that there exists a positive constant $C$
such that 
\[\frac{1}{C}g(\lambda)\leq f(\lambda)\leq Cg(\lambda), \qquad\lambda\in E.\]
Given $\lambda_0\in\mathbb{T}$, we say that $f\sim g$ as $\lambda\to\lambda_0$ if 
\[\lim_{\lambda\to\lambda_0}\frac{f(\lambda)}{g(\lambda)}=1.\]
\begin{proof}[Proof of Theorem \ref{T:Hbexp}]
    The function $b$ is entire and in $H^\infty_1$. It is also non-extreme, since
    \[\log(1-|b(\zeta)|)=\log(1-e^{\operatorname{Re}(\zeta^N)-1}), \qquad \zeta\in\mathbb{T},\]
    and therefore
    \[\int_{\mathbb{T}}\log(1-|b|)\,\operatorname{d}\!m=\frac{1}{2\pi} \int_0^{2\pi}\log(1-e^{\cos(N\theta)-1})\,\operatorname{d}\!\theta>-\infty.\]
    We have that 
    \[\sigma(b)=\{\zeta\in\mathbb{T}\colon |b(\zeta)|<1\}=\mathbb{T}\setminus\{\xi_1,\ldots,\xi_N\},\]
and we check that the condition of Theorem \ref{T:applicationbcontinuous} holds. Notice that, for every $k=1,\ldots,N$, writing $\lambda=e^{i\theta}\in\mathbb{T}$, we have the following uniform estimates in a neighborhood of $\xi_k=e^{\frac{2k\pi i}{N}}$:
    \begin{align*}
       \frac{1-|b(\lambda)|^2}{\prod_{j=1}^N|\lambda-\zeta_j|^2}\asymp \frac{1-e^{2\operatorname{Re}(\lambda^N)-2}}{|\lambda-\zeta_k|^2}=\frac{1-e^{2\cos(N\theta)-2}}{2-2\cos(\theta-\frac{2k\pi}{N})}.
    \end{align*}
    Now, as $\theta$ goes to $\frac{2k\pi}{N}$, we have that
    \[\frac{1-e^{2\cos(N\theta)-2}}{2-2\cos(\theta-\frac{2k\pi}{N})} \sim \frac{N^2(\theta-\frac{2k\pi}{N})^2}{(\theta-\frac{2k\pi}{N})^2}=N^2.\]
    In particular, we showed that
   \[0<\inf_{\lambda\in\sigma(b)}\frac{1-|b(\lambda)|^2}{\prod_{j=1}^N|\lambda-\zeta_j|^2}\leq \sup_{\lambda\in\sigma(b)}\frac{1-|b(\lambda)|^2}{\prod_{j=1}^N|\lambda-\zeta_j|^2} <\infty,\]
   and we conclude using Theorem \ref{T:applicationbcontinuous} that $H(b)=\mathcal{D}_\mu.$
\end{proof}

\bibliographystyle{plain}
\bibliography{mybibliography}

\end{document}